\documentclass[12pt]{article}

\usepackage{amsmath}
\usepackage{amsfonts}
\usepackage{latexsym}
\usepackage{amssymb}
\usepackage{curves}
\usepackage{graphicx}
\usepackage{wasysym}
\usepackage{MnSymbol}
\usepackage{amsthm,amscd}

\newtheorem{thm}{Theorem}[section]
\newtheorem{pro}[thm]{Proposition}
\newtheorem{cor}[thm]{Corollary}
\newtheorem{lem}[thm]{Lemma}

\newcommand{\mapsfrom}
{\mathrel{\reflectbox{\ensuremath{\mapsto}}}}

\newcommand{\noin}{\noindent}

\newcommand{\NN}{{\mathbb{N}}}

\newcommand{\QQ}{{\mathbb{Q}}}

\newcommand{\openbin}{\left( \!\! \begin{array}{c}}
\newcommand{\closebin}{\end{array} \!\! \right)}
\newcommand{\openvec}{\left[ \!\! \begin{array}{c}}
\newcommand{\closevec}{\end{array} \!\! \right]}
\newcommand{\openmat}{\left[ \!\! \begin{array}{cc}}
\newcommand{\closemat}{\end{array} \!\! \right]}
\newcommand{\opentri}{\left[ \!\! \begin{array}{ccc}}
\newcommand{\closetri}{\end{array} \!\! \right]}
\newcommand{\openquad}{\left[ \!\! \begin{array}{cccc}}
\newcommand{\closequad}{\end{array} \!\! \right]}
\newcommand{\openquin}{\left[ \!\! \begin{array}{ccccc}}
\newcommand{\closequin}{\end{array} \!\! \right]}
\newcommand{\opensex}{\left[ \!\! \begin{array}{cccccc}}
\newcommand{\closesex}{\end{array} \!\! \right]}

\newcommand{\End}{\mbox{\rm{End}}}

\newcommand{\cod}{\mbox{\rm{cod}}}

\newcommand{\dom}{\mbox{\rm{dom}}}

\newcommand{\id}{{\rm{id}}}

\newcommand{\sspan}{\mbox{\rm{span}}}

\newcommand{\hh}{\widehat}
\newcommand{\oo}{\overline}

\newcommand{\cA}{{\cal A}}
\newcommand{\cB}{{\cal B}}
\newcommand{\cC}{{\cal C}}
\newcommand{\cD}{{\cal D}}

\newcommand{\cG}{{\cal G}}

\newcommand{\cI}{{\cal I}}

\newcommand{\cL}{{\cal L}}

\newcommand{\cP}{{\cal P}}

\newcommand{\cR}{{\cal R}}
\newcommand{\cS}{{\cal S}}

\newcommand{\cX}{{\cal X}}

\begin{document}

%\small

\title{Some deformations of the fibred
biset category}

\author{\large Laurence
Barker\footnote{e-mail: barker@fen.bilkent.edu.tr,
Some of this work was done while this author was
on sabbatical leave, visiting the Department of
Mathematics at City, University of London.}
\hspace{1in} \.{I}smail Alperen
\"{O}\u{g}\"{u}t\footnote{e-mail: ismail.ogut@bilkent.edu.tr.} \\
\mbox{} \\
Department of Mathematics \\
Bilkent University \\
06800 Bilkent, Ankara, Turkey \\
\mbox{}}

\maketitle

\small
\begin{abstract}
\noin We prove the well-definedness of
some deformations of the fibred biset category
in characteristic zero. The method is to realize
the fibred biset category and the deformations
as the invariant parts of some categories whose
compositions are given by simpler formulas.
Those larger categories are constructed from
a partial category of subcharacters by
linearizing and introducing a cocycle.

\smallskip
\noin 2010 {\it Mathematics Subject Classification:}
Primary 19A22, Secondary 16B50.

\smallskip
\noin {\it Keywords:} partial category; linear
category; subgroup category; star product;
subcharacter

\end{abstract}

%Section 1
\section{Introduction}
\label{1}

One approach to finite group theory involves
linear categories whose objects are finite
groups. Examples include the biset category
studied in Bouc \cite{Bou10}, the fibred biset
category in Boltje--Co\c{s}kun \cite{BC18},
the $p$-permutation category in Ducellier
\cite{Duc16} and many subcategories of
those. The work behind the present paper has
been an attempt, in some cases successful,
to characterize such categories in terms of
categories that are larger but easier to
describe. For the biset category, the theme
was initiated in Boltje--Danz \cite{BD13} and
developed in \cite{BO}. Our presentation,
though, is self-contained and does not
presume familiarity with those two papers.

Throughout, we let $\cG$ be a non-empty set
of finite groups. It is always to be understood
that $F$, $G$, $H$, $I$ denote arbitrary
elements of $\cG$. We let $R$ be a commutative
unital ring such that every positive integer has
an inverse in $R$. The inversion condition,
expressed differently, is that the field
of rational numbers $\QQ$ embeds in $R$.
We let $K$ be an algebraically closed field
of characteristic zero. We let $A$ be a
multiplicatively written abelian group.

After reviewing some background in Section
\ref{2}, we shall introduce the notion of
an interior $R$-linear category $\cL$ with
set of objects $\cG$. Each $G$ acts on
the endomorphism algebra $\End_\cL(G)$
via an algebra map from the group algebra
$RG$. We shall construct a category
$\oo{\cL}$, called the invariant
category of $\cL$.

Informally, borrowing a term from
algebraic geometry, we call $\cL$
a ``polarization'' of $\oo{\cL}$. Let us
retain the scare-quotes, because we
do not propose a general definition, and
we wish only to use the term when the
composition for $\cL$ is easier to
describe than the composition for
$\oo{\cL}$. A ``polarization'' of the
biset category was introduced in
\cite{BD13}, and that was extended to some
deformations of the biset category in \cite{BO}.
In Section \ref{3}, as rather a toy illustration,
we shall introduce a ``polarization'' of a
$K$-linear category associated with
$K$-character rings. More substantially,
in Section \ref{4}, we shall introduce a partial
category called the $A$-subcharacter partial
category and, in Section \ref{5}, we shall
show that a twisted $R$-linearization of the
$A$-subcharacter partial category serves
as a ``polarization'' of the $R$-linear
$A$-fibred biset category discussed in
Boltje--Co\c{s}kun \cite{BC18}. One direction for
further study may be towards reassessing
the classification, in \cite{BC18}, of the simple
$A$-fibred biset functors. We shall comment
further on that at the end of the paper.

Also in Section \ref{5}, we shall present some
deformations of the $R$-linear $A$-fibred
biset category. To prove the associativity
of the deformed composition, we shall
make use of the fact that those
deformations, too, admit ``polarizations''
in the form of twisted $R$-linearizations
of the $A$-subcharacter partial category.

Our hypothesis on $R$ is not
significantly more general than the
case of an arbitrary field of characteristic
zero. Adaptations to other coefficient
rings would require further techniques.

%Section 2
\section{Interior linear categories}
\label{2}

Categories and partial categories arise in
our topic mainly as combinatorial structures
(in the sense that some familar ``up to''
qualifications are absent, to wit, all the
equivalences of categories below are
isomorphisms of categories). Let us
organize our notation and terminology
accordingly. The idea behind the less
standard among the following definitions
is not new. It goes back at least as far as
Schelp \cite{Sch72}. For clarity, let us
present the material in a self-contained
way. We define a {\bf partial magma} to
be a set $\cP$ equipped with a relation
$\sim$, called the {\bf matching relation},
together with a function $\cP \ni \phi \psi
\mapsfrom (\phi, \psi) \in \Gamma(\cP)$,
called the {\bf multiplication}, where
$\Gamma(\cP) = \{ (\phi, \psi) \in
\cP \times \cP : \phi \sim \psi \}$.

We call $\cP$ a {\bf partial semigroup}
provided the following associativity condition
holds: given $\theta, \phi, \psi \in \cP$ such
that $\theta \sim \phi$ and $\phi \sim \psi$,
then $\theta \sim \phi \psi$ if and only if
$\theta \phi \sim \psi$, in which case
$\theta (\phi \psi) = (\theta \phi) \psi$.
When $\theta \sim \phi \psi$, we say that
$\theta \phi \psi$ is defined.

Suppose $\cP$ is a partial semigroup.
An element $\iota \in \cP$ satisfying
$\iota \sim \iota$ and $\iota^2 = \iota$ is
called an {\bf idempotent} of $\cP$. Let
$\cX$ be a set and $\cI = (\id_X^\cP :
X \in \cX)$ a family of idempotents
$\id_X^\cP \in \cP$ satisfying the
following filtration condition: for all $\phi
\in \cP$, we have $\id_X^\cP \sim \phi
\sim \id_Y^\cP$ for unique $X, Y \in \cX$,
furthermore, $\id_X^\cP \phi = \phi =
\phi \iota_Y^\cP$. We write $\cod(\phi)
= X$ and $\dom(\phi) = Y$, which we
call the {\bf codomain} and {\bf domain}
of $\phi$, respectively. We call the triple
$(\cP, \cI, \cX)$ a {\bf small partial
category} on $\cX$. As an abuse of
notation, we often write $\cP$ instead
of $(\cP, \cI, \cX)$. We call an element
$\phi \in \cP$ a {\bf $\cP$-morphism}
$\cod(\phi) \leftarrow \dom(\phi)$, we
call an element $X \in \cX$ an
{\bf object} of $\cP$ and we call
$\id_X^\cP$ the {\bf identity
$\cP$-morphism} on $X$. We write
$$\cP(X, Y) = \{ \phi \in \cP :
  \cod(\phi) = X,  \, \dom(\phi) = Y \}$$
and $\End_\cP(X) = \cP(X, X)$. In the
context of partial categories, products are
called {\bf composites}. Observe that, given
$\cP$-morphisms $\phi$ and $\psi$ such
that $\phi \sim \psi$, then $\dom(\phi) =
\cod(\psi)$. If, conversely, $\phi \sim \psi$
for all $\cP$-morphisms $\phi$ and $\psi$
satisfying $\dom(\phi) = \cod(\psi)$, then
we call $\cP$ a {\bf small category}.
Of course, the latest definition coincides
with the usual definition of the same term;
a small category in the above sense
determines all the structural features of
a small category in the conventional sense,
and conversely.

Another approach to the above material is
as follows, directly generalizing the notion
of a small category expressed in Bourbaki
\cite[II Section 3 D\'{e}finition 2]{Bou16}.
A small partial category
$(\cP, \cI, \cX)$ uniquely determines a
quiver equipped with a composition
operation, where $\cX$, $\cP$,
$\dom()$, $\cod()$ are the vertex set,
arrow set, source function, target
function, respectively, and the
composition operation $\cP \leftarrow
\Gamma(\cP)$ satisfies evident
versions of the associativity and
identity axioms. The reason for our
treatment using semigroups rather
than quivers is that the former will
be more convenient when discussing
the algebras $R \cP$ and
$R_\gamma \cP$, defined below.
One special case worth bearing in
mind will be that where $\cP$ is a
group, whereupon $R_\gamma \cP$
is a twisted group algebra.

All the categories and partial categories
discussed below are deemed to be
small, and we shall omit the term
{\it small}, though the material
extends easily to locally small cases.

Given categories $\cC$ and $\cD$ on
a set $\cX$, then a functor $\lambda :
\cC \leftarrow \cD$ is said to be
{\bf object-identical} provided
$\lambda(X) = X$ for all $X \in \cX$. Note
that, if such $\lambda$ is an equivalence,
then $\lambda$ is an isomorphism.

Recall, a category is said to be
{\bf $R$-linear} when the morphism sets
are $R$-modules and the composition
maps are $R$-bilinear. Functors between
$R$-linear categories are required to
be $R$-linear on morphisms. We define
an {\bf interior $R$-linear category} on
$\cG$ to be an $R$-linear category
$\cL$ on $\cG$ equipped with a family
$(\sigma_G)$ of algebra maps
$$\sigma_G \: : \: \End_\cL(G)
  \leftarrow RG$$
called the {\bf structural maps} of $\cL$.
We write elements of $F {\times} G$ in
the form $f {\times} g$ instead of the
conventional $(f, g)$ (because the
unconventional notation is the more
readable when familiarity has been
acquired). We make $\cL(F, G)$ become
an $R(F {\times} G)$-module such that
$f {\times} g$ sends an element
$\phi \in \cL(F, G)$ to the element
$${}^{f \times g} \phi = \sigma_F(f) \,
  \phi \, \sigma_G(g)^{-1} \; .$$
We write ${}^f \phi {}^g = {}^{(f \times
g^{-1})} \phi$ and we also use the
notation ${}^f \phi = {}^f \phi {}^1$
and $\phi {}^g = {}^1 \phi {}^g$.

\begin{pro}
\label{2.1} Given an interior $R$-linear
category $\cL$ on $\cG$, then there is an
$R$-linear category $\oo{\cL}$ on $\cG$
such that, for all $F, G \in \cG$, the
$R$-module of $\oo{\cL}$-morphisms
$F \leftarrow G$ is the
$F {\times} G$-fixed $R$-submodule
$$\oo{\cL}(F, G) =
  \cL(F, G)^{F \times G}$$
and the composition for $\oo{\cL}$ is
restricted from the composition for $\cL$.
\end{pro}

\begin{proof}
We define
${\displaystyle e_G = \frac{1}{|G|}
\sum_{g \in G} g }$
which is an idempotent $Z(RG)$. We have
$$\oo{\cL}(F, G) = \sigma_F(e_F) \,
  \cL(F, G) \, \sigma_G(e_G) \; .$$
So $\oo{\cL}$ is a category as specified,
with identity morphisms $\id_G^{\oo{\cL}}
= \sigma_G(e_G)$.
\end{proof}

We call $\oo{\cL}$ the {\bf invariant
category} of $\cL$. Note that $\oo{\cL}$
need not be a subcategory of $\cL$,
since $\id_G^{\oo{\cL}}$ may be distinct
from $\id_G^\cL$.

We define the {\bf $R$-linearization} of a partial
semigroup $\cP$ to be the algebra $R \cP$
over $R$ such that $R \cP$ is freely generated
over $R$ by $\cP$ and the multiplication on
$R \cP$ is given by $R$-linear extension of
the multiplication for $\cP$, with the
understanding that $\phi \psi = 0$ whenever
$\phi \not\sim \psi$. Let $R^\times$ denote
the unit group of $R$. We define a {\bf cocycle}
for $\cP$ over $R$ to be a function
$\gamma : R \leftarrow \cP \times \cP$
satisfying the following two conditions:

\smallskip
\noin {\bf Non-degeneracy:} Given
$\phi, \psi \in \cP$, then $\gamma(\phi, \psi)
\in R^\times$ if $\phi \sim \psi$, whereas
$\gamma(\phi, \psi) = 0$ if
$\phi \not\sim \psi$.

\noin {\bf Associativity:} Given
$\theta, \phi, \psi \in \cP$ with
$\theta \phi \psi$ defined, then
$\gamma(\theta, \phi) \gamma(\theta \phi,
\psi) = \gamma(\theta, \phi \psi)
\gamma(\phi, \psi)$.

\smallskip
\noin Fixing $\gamma$, let $R_\gamma \cP$
be the $R$-module freely generated by the
set of formal symbols $\{ p_\phi :
\phi \in \cP \}$. We make $R_\gamma \cP$
become an (associative, not necessarily
unital) algebra over $R$ by taking the
multiplication to be such that
$$p_\phi p_\psi =
  \gamma(\phi, \psi) p_{\phi \psi} \; .$$
We call $R_\gamma \cP$ the {\bf twisted
linearization} of $\cP$ with cocycle
$\gamma$. When $\gamma(\phi, \psi) = 1$
for all $(\phi, \psi) \in \Gamma(\cP)$, we call
$\gamma$ the {\bf trivial cocycle} for $\cP$.
In that case, we have an algebra isomorphism
$R_\gamma \cP \cong \cP$ given by
$p_\phi \leftrightarrow \phi$.

In later sections, we shall be
considering scenarios having the following
form. Suppose, now, that $\cP$ is a partial
category on $\cG$. Thus, we are supposing
that $\cP$ comes equipped with a family of
idempotents $(\id_G^\cP)$ satisfying the
filtration condition. It is easy to see that the
$R$-linearization $R \cP$ is an $R$-linear
category and $\id_G^{R \cP} = \id_G^\cP$.
Assume also that $R \cP$ is equipped
with the structure of an interior $R$-linear
category such that, for all $F, G \in \cG$, the
action of $F {\times} G$ on $R \cP(F, G)$
restricts to an action on $\cP(F, G)$. Define
$$\oo{\phi} = \sigma_F(e_F) \phi
  \sigma_G(e_G) = \frac{1}{|F| . |G|}
  \sum_{f \in F, g \in G}
  {}^f \! \phi {}^g$$
for $\phi \in \cP(F, G)$. Note that $\oo{\phi}
= \oo{{}^f \! \phi {}^g}$ and, if we let $\phi$
run over representatives of the
$F {\times} G$-orbits in $\cP(F, G)$, then
$\oo{\phi}$ runs over the elements of an
$R$-basis for $\oo{R \cP}(F, G)$. We have
$$\oo{\phi} \oo{\psi} = \sigma_F(e_F) \phi
  \sigma_G(e_G) \psi \sigma_H(e_H) =
  \frac{1}{|G|} \sum_{g \in G} \sigma_F(e_F)
  \phi . {}^g \psi \sigma_H(e_H) =
  \frac{1}{|G|} \sum_{g \in G}
  \oo{\phi \, . \, {}^g \psi}$$
for all $\phi \in \cP(F, G)$ and
$\psi \in \cP(G, H)$, the dot in the
formula inserted only for readability.
Similar comments hold for the twisted
linearizations. Let us make those comments,
because some modification is needed. Let
$\gamma$ be a cocycle for the partial
category $\cP$. To confirm that the twisted
$R$-linearization $R_\gamma \cP$ is an
$R$-linear category, observe that, writing
$\iota = \id_G^\cP$, then
$\gamma(\phi, \iota) =
\gamma(\iota, \iota)$, whence
$$\id_G^{R_\gamma \cP} =
  \gamma(\iota, \iota)^{-1} p_\iota \; .$$
Assume now that the structure of an interior
$R$-linear category is imposed on
$R_\gamma \cP$ instead of $R \cP$,
furthermore, for all $F, G \in \cG$, the
action of $F {\times} G$ on $R_\gamma
\cP(F, G)$ restricts to the action on
$\cP(F, G)$ and each
$\gamma(\phi^g, {}^g \psi) =
\gamma(\phi, \psi)$. Again, the elements
$$\oo{p}_\phi = \sigma_F(e_F) p_\phi
  \sigma_G(e_G) = \frac{1}{|F| . |G|}
  \sum_{f \in F, g \in G}
  {}^f \! (p_\phi) {}^g$$
comprise an $R$-basis for
$\oo{R_\gamma \cP}(F, G)$. A
manipulation similar to that for
$\oo{\phi} \oo{\psi}$ yields
$$\oo{p}_\phi \oo{p}_\psi =
  \frac{1}{|G|} \sum_{g \in G}
  \gamma(\phi, {}^g \psi) \,
  \oo{p}_{\phi . {}^g \psi} \; .$$

%Section 3
\section{The ordinary character category}
\label{3}

After Romero \cite[Section 4]{Rom12},
whose study was in the richer context
of Green biset functors, we shall describe
a $K$-linear category $K \cA_K$ associated with ordinary
$K$-character rings of finite groups. We
shall then realize $K \cA_K$ as the
invariant category $\oo{K \cR}$ of an
interior $K$-linear category $K \cR$.

For a finite group $E$, we write
$\cA_K(E)$ to denote the ring of
$K$-characters of $E$. That is to say,
$\cA_K(E)$ is the Grothendieck ring of
the category of finitely generated
$KE$-modules. Incidentally, the
multiplication on
$\cA_K(E)$ is given by tensor product
over $K$, but we shall not be making
use of that. Given a $KE$-module $M$,
we identify the isomorphism class of
$M$ with the $K$-character $\chi_K :
K \leftarrow E$ of $M$. Thus, $\cA_K(E)$
has a basis consisting of the irreducible
$K$-characters of $E$. The $K$-linear
extension $K \cA_K(E)$ can be
identified with the $K$-module of
class functions $K \leftarrow E$.

Any $KF$-$KG$-bimodule $X$ can
be regarded as a $K(F {\times} G)$-module
by writing $fxg^{-1} = (f {\times} g)x$ for
$f \in F$, $g \in G$, $x \in X$. In
particular, the isomorphism class of
$X$ can be identified with the
$K$-character $\chi_X : K \leftarrow
F {\times} G$.
We form the $K$-linear category
$K \cA_K$ with morphism $K$-modules
$K \cA_K(F, G) = K \cA_K(F {\times} G)$
and with composition $K \cA_K(F, H)
\leftarrow K \cA_K(F, G) \times
K \cA(G, H)$ such that, given a
$KF$-$KG$-bimodule $X$ and a
$KG$-$KH$-bimodule $Y$, writing
$Z = X \otimes_{KG} Y$, then the
composite is $\chi_X \chi_Y = \chi_Z$.
The next result, from Bouc
\cite[7.1.3]{Bou10}, describes the
composition more explicitly. Let
us give a quick alternative proof.

\begin{lem}
\label{3.1}
Let $\xi \in K \cA(F, G)$ and
$\eta \in K \cA(G, H)$. Let $f \in F$
and $h \in H$. Then
$$(\xi \eta)(f {\times} h) =
  \frac{1}{|G|} \sum_{g \in G}
  \xi(f {\times} g) \eta(g {\times} h) \; .$$
\end{lem}

\begin{proof}
Let $X$, $Y$, $Z$ be as above. By
$K$-linearity, we may assume that
$\xi = \chi_X$ and $\eta = \chi_Y$.
Let $\zeta = \chi_Z$. Then
$\zeta(f {\times} h) =
(\xi \eta)(f {\times} h)$. Let $\hh{Z}
= X \otimes_K Y$ regarded as a
module of $K(F {\times} G {\times} H)$
such that
$$(f {\times} g {\times} h)(x \otimes_K y)
  = f x g^{-1} \otimes_K g y h^{-1}
  = (f {\times} g) x \otimes_K
  (g {\times} h) y$$
for $g \in G$, $x \in X$, $y \in Y$. Then
$\chi_{\hh{Z}}(f {\times} g {\times} h)
= \xi(f {\times} g) \eta(g {\times} h)$.
Let $F {\times} H$ and $G$ act on
$\hh{Z}$ via the canonical embeddings
in $F {\times} G {\times} H$. As a direct
sum of $K(F {\times} H)$-modules,
$\hh{Z} = \hh{Z}^G \oplus
\hh{Z}_{(G)}$, where $\hh{Z}^G$
denotes the $G$-fixed submodule and
$$\hh{Z}_{(G)} = \sspan_K \{ x g^{-1}
  \otimes_K gy - x \otimes_K y \} =
  \sspan_K \{ x g^{-1} \otimes_K y
  - x \otimes_K gy \} \; .$$
So the $G$-cofixed quotient
$\hh{Z}_G = \hh{Z} / \hh{Z}_{(G)}$
is a $K(F {\times} H)$-module and
$Z \cong \hh{Z}_G \cong
\hh{Z}^G = e_G \hh{Z}$.
Therefore, the trace of the action of
$f {\times} h$ on $Z$ is equal to the
trace of the action of $\sum_g
(f {\times} g {\times} h)/|G|$ on
$\hh{Z}$. That is to say, $\zeta(f
{\times} h) = \sum_g \chi_{\hh{Z}}(f
{\times} g {\times} h) / |G|$.
\end{proof}

Let $\cR$ be the partial category on $\cG$
such that $\cR(F, G) = F {\times} G$ and,
given $u {\times} v \in \cR(F {\times} G)$
and $v' {\times} w \in \cR(H {\times} G)$,
then $(u {\times} v) \sim (v' {\times} w)$
if and only if $v = v'$, furthermore,
$(u {\times} v)(v {\times} w) =
u {\times} w$. We make the
$K$-linearization $K \cR$ become an
interior $K$-linear category by defining
$$\sigma_G(g) = \sum_{v \in G}
  ({}^g v) {\times} v$$
where ${}^g v = g v g^{-1}$. Thus,
$F {\times} G$ acts on $\cR(F, G)$ by
${}^f \! (u {\times} v)^g = ({}^f \! u)
{\times} (v^g)$, where $v^g =
g^{-1} v g$. Let
$$\mu_{F, G} \: : \: \oo{K \cR}(F, G)
  \leftarrow K \cA_K(F, G)$$
be the $K$-linear map given by
$$\mu_{F, G}(\xi) = \frac{1}{|F|}
  \sum_{u \in F, v \in G} \xi(u {\times} v)
  \, \oo{u {\times} v} \; .$$

\begin{pro}
\label{3.2}
The maps $\mu_{F, G}$, for
$F, G \in \cG$, determine an
object-identical isomorphism of
$K$-linear categories $\mu : \oo{K \cR}
\leftarrow K \cA_K$.
\end{pro}

\begin{proof}
For $u {\times} v \in \cR(F {\times} G)$,
let $\xi_{\oo{u \times v}}$ be the element
of $K \cA_K(F, G)$ such that, given
$u_1 {\times} v_1 \in F {\times} G$, then
$\xi_{\oo{u \times v}}(u_1 {\times} v_1)
= 1$ when $u {\times} v$ and
$u_1 {\times} v_1$ are
$F {\times} G$-conjugate, otherwise
$\xi_{\oo{u {\times} v}}(u_1 {\times}
v_1) = 0$. Letting $u {\times} v$ run
over representatives of the conjugacy
classes of $F {\times} G$, then
$\oo{u {\times} v}$ runs over the
elements of a $K$-basis for
$\oo{K \cR}(F, G)$, while
$\xi_{\oo{u {\times} v}}$ runs over
the elements of a $K$-basis for
$K \cA_K(F, G)$. We have
$$|[u {\times} v]_{F \times G}| \,
  \oo{u {\times} v} = |F| \, \mu_{F \times G}
  (\xi_{\oo{u \times v}}) \; .$$
So $\mu_{F, G}$ is a $K$-isomorphism.

Now fix $u {\times} v \in \cR(F
{\times} G)$ and $v' {\times} w \in
\cR(G {\times} H)$. Write $[v]_G$
for the $G$-conjugacy class of $v$.
Substituting $\oo{\phi} =
\oo{u {\times} v}$ and $\oo{\psi}
= \oo{v' {\times} w}$, the general
formula for $\oo{\phi} \oo{\psi}$
in Section \ref{2} becomes
$$\oo{u {\times} v} .
  \oo{v' {\times} w} = \frac{1}{|G|}
  \sum_{g \in G} \oo{(u {\times} v)
  (({}^g v') {\times} w)} \; .$$
So $\oo{u {\times} v} . \oo{v'
{\times} w} = 0$ unless $[v]_G
= [v']_G$, in which case,
$\oo{u {\times} v} . \oo{v'
{\times} w} = \oo{u {\times} w} \,
|C_G(v)| / |G| = \oo{u {\times} w}
/ |[v]_G|$. Therefore,
$$\mu_{F, G}(\xi) \mu_{G, H}(\eta)
  = \frac{1}{|F|.|G|} \sum_{u \times
  v \in F \times G, v' \times w \in
  G \times H} \xi(u {\times} v)
  \eta(v' {\times} w) \, \oo{u {\times} v}
  . \oo{v' {\times} w}$$
$$= \frac{1}{|F|.|G|} \sum_{u \in F,
  v \in G, w \in H} \xi(u {\times} v)
  \eta(v {\times} w) \, \oo{u {\times} w}
  = \mu_{F, H}(\xi \eta)$$
for all $\xi$ and $\eta$ as in
Lemma \ref{3.1}.
\end{proof}

%Section 4
\section{The subcharacter partial category}
\label{4}

We shall introduce a category $\cS^A$ on
$\cG$, called the {\bf $A$-subcharacter
partial category} on $\cG$. We shall
construct a twisted $R$-linearization
$R_\ell \cS^A$ of $\cS^A$
parameterized by a multiplicative
monoid homomorphism $\ell : R^\times
\leftarrow \NN - \{ 0 \}$. After equipping
$R_\ell \cS^A$ with structural maps to
make $R_\ell \cS^A$ become an interior
$R$-linear algebra, we shall explicitly
describe the invariant category
$\oo{R_\ell \cS^A}$. That description
will be applied to deformations of the
$R$-linear $A$-fibred biset category
in the next section. Some of our
terminology and notation is adapted
from \cite{Bar04}, \cite{BO} and
Boltje--Co\c{s}kun \cite{BC18}, but
our account is self-contained.

To introduce some notation that we
shall be needing, let us review the
definition of the subgroup category
$\cS$ on $\cG$. (The category would
be written as $\cS_\cG$ in the notation
of \cite{BO}.) Consider the groups
$F, G, H, I \in \cG$. We let $\cS(F, G)$
denote the set of subgroups of
$F {\times} G$. Let $U \in \cS(F, G)$,
$V \in \cS(G, H)$, $W \in \cS(H, I)$.
We define
$$\Gamma(U, V) = \{ f {\times} g
  {\times} h : f {\times} g \in U,
  g {\times} h \in V \} \; .$$
After Bouc \cite[2.3.19]{Bou10}, we
define the {\bf star product} $U * V
\in \cS(F, H)$ to be
$$U * V = \{ f {\times} h : f {\times} g
  {\times} h \in \Gamma(U, V) \} \; .$$
Plainly, $*$ is associative. We point out
that, defining
$$\Gamma(U, V, W) = \{ f {\times} g
  {\times} h {\times} i : f {\times} g \in U,
  g {\times} h \in V, h {\times} i \in W \}$$
then $U * V * W = \{ f {\times} i :
f {\times} g {\times} h {\times} i \in
\Gamma(U, V, W) \}$. We make $\cS$
become a category by taking the
composition to be star product.

Below, when we have established the
construction of the partial category
$\cS^A$, it will be clear that $\cS^A$
coincides with $\cS$ when $A$ is
trivial. First, though, we need the
patience for a few definitions. As
in Boltje--Danz \cite{BD13}, we write
$p_2(U)$ and $p_1(V)$, respectively,
for the images of the projections of
$U$ and $V$ to $G$. We define
$k_2(U) = \{ g : 1 {\times} g \in U \}$
and $k_1(V) = \{ g : g {\times} 1
\in V \}$. Let
$$\Gamma_\cap(U, V) =
  k_1(U) \cap k_2(V) =
  \{ g \in G : 1 {\times} g {\times} 1
  \in \Gamma(U, V) \} \; .$$
The following lemma is part of
\cite[3.5]{BD13}.

\begin{lem}
\label{4.1} {\rm (Boltje--Danz.)}
With the notation above,
$$|\Gamma_\cap(U, V)| .
  |\Gamma_\cap(U * V, W)| =
  |\Gamma_\cap(U, V * W)| .
  |\Gamma_\cap(V, W)| \; .$$
\end{lem}

\begin{lem}
\label{4.2}
With the notation above, $|U| . |V|
= |p_2(U) p_1(V)| .
|\Gamma_\cap(U, V)| . |U * V|$.
\end{lem}

\begin{proof}
Let $\Gamma = \{ f {\times} g
{\times} h : f {\times} g \in U,
g {\times} h \in V \}$ and
$$\Lambda = p_2(U) \cap p_1(V)
  = \{ g : \big( \exists f {\times} h
  \in F {\times} H \big) \big(
  f {\times} g {\times} h \in
  \Gamma \big) \} \; .$$
Observe that $|\Lambda| = |p_2(U)|
. |p_1(V)| / |p_2(U) p_1(V)|$. Fix
$f {\times} g {\times} h \in \Gamma$.
Given $g' \in G$, then $f {\times} g'
{\times} h \in \Gamma$ if and only if
$g' g^{-1} \in \Gamma_\cap(U, V)$. So
$$|\Gamma| = |\Gamma_\cap(U, V)|
  . |U * V| \; .$$
Meanwhile, given
$f' {\times} h' \in F {\times} H$, then
$f' {\times} g {\times} h' \in \Gamma$
if and only if $f' f^{-1} \in k_1(U)$ and
$h' h^{-1} \in k_2(V)$. So $|\Gamma|
= |\Lambda| . |k_1(U)| . |k_2(V)|$. Since
$|U| = |p_2(U)| . |k_1(U)|$ and
$|V| = |p_1(V)| . |k_2(V)|$, we have
$$|\Gamma| = \frac{|\Lambda| .
  |U| . |V|}{|p_2(U)| . |p_1(V)|} =
  \frac{|U| . |V|}{|p_2(U) p_1(V)|} \; .$$
Eliminating $\Gamma$, we obtain
the required equality.
\end{proof}

For a finite group $E$, we define an
{\bf $A$-character} of $E$ to be a
homomorphism $A \leftarrow E$. We define
an {\bf $A$-subcharacter} to be a pair
$(T, \tau)$ consisting of a subgroup $T$
of $E$ and an $A$-character $\tau$ of
$E$. The set $\cS^A(E)$ of
$A$-subcharacters of $E$ becomes
an $E$-set via the conjugation actions
of $E$ on the two coordinates, that is,
given $g \in G$ and $t \in T$,
then ${}^g (T, \tau) = ({}^g T, {}^g \tau)$
where ${}^g \tau({}^g t) = \tau(t)$. When
$E$ is understood from the context, we
write $[T, \tau]$ to denote the $E$-orbit
of $(T, \tau)$. We write $\cS^A[E]$ to
denote the set of $E$-orbits in
$\cS^A(E)$.

Define $\cS^A(F, G) =
\cS^A(F {\times} G)$ and
$\cS^A[F, G] = \cS^A[F {\times} G]$. Let
$(U, \mu)$, $(V, \nu)$, $(W, \omega)$
be $A$-subcharacters in $\cS^A(F, G)$,
$\cS^A(G, H)$, $\cS^A(H, I)$,
respectively. We write $(U, \mu) \sim
(V, \nu)$ provided $\mu(1 {\times} g)
\nu(g {\times} 1) = 1$ for all
$g \in \Gamma_\cap(U, V)$. When
that condition holds, we define
$\mu * \nu$ to be the
$A$-character of $U * V$ given by
$(\mu * \nu)(f {\times} h) =
\mu(f {\times} g) \nu(g {\times} h)$
for $f {\times} g {\times} h \in
\Gamma(U, V)$.

\begin{pro}
\label{4.3}
Defining composition by
$(U, \mu) * (V, \nu) = (U * V, \mu * \nu)$
when $(U, \mu) \sim (V, \nu)$, then
$\cS^A$ becomes a partial category.
\end{pro}

\begin{proof}
We claim that the conditions

\smallskip
\noin $\bullet$ $(U, \mu) \sim
(V, \nu)$ and $(U * V, \mu * \nu) \sim
(W, \omega)$,

\noin $\bullet$ $(V, \nu) \sim
(W, \omega)$ and $(U, \mu) \sim
(V * W, \nu * \omega)$,

\smallskip
\noin are equivalent and, when they
hold, $(\mu * \nu) * \omega = \mu *
(\nu * \omega)$. It is straightforward
to confirm that the two conditions are
equivalent to:

\smallskip
\noin  $\bullet$ for all $g {\times} h \in
G {\times} H$ satisfying $1 {\times} g
{\times} h {\times} 1 \in \Gamma(U,
V, W)$, we have $\mu(1{\times} g)
\nu(g {\times} h) \omega(h {\times} 1) = 1$.

\smallskip
\noin Plainly, when the three equivalent
conditions hold, the expression $\mu *
\nu * \omega$ is unambiguous and
$$(\mu * \nu * \omega)(f {\times} i) =
  \mu(f {\times} g) \nu(g {\times} h)
  \omega(h {\times} i)$$
for all $f {\times} g {\times} h {\times} i
\in \Gamma(U, V, W)$. The claim is
established. To finish the proof, we
observe that $\id_G^{\cS^A} =
(\Delta(G), 1)$, where
$\Delta(G) = \{ y {\times} y : y \in G \}$
and $1$ denotes the trivial $A$-character.
\end{proof}

By Lemma \ref{4.1} there is a cocycle
$\gamma_\ell$ for $\cS^A$ given by
$$\gamma_\ell((U, \mu), (V, \nu)) =
  \ell(|\Gamma_\cap(U, V)|)$$
when $(U, \mu) \sim (V, \nu)$. Note,
the condition that $\gamma_\ell$ is
a cocycle implies that
$\gamma_\ell((U, \mu),
(V, \nu)) = 0$ when
$(U, \mu) \not\sim (V, \nu)$. We
define $R_\ell \cS^A =
R_{\gamma_\ell} \cS^A$. Thus,
$$R_\ell \cS^A(F, G) = \bigoplus_{(U, \mu)
  \in \cS^A(F, G)} R \, s_{U, \mu}^{F, G}$$
as a direct sum of regular $R$-modules,
where $s_{U, \mu}^{F, G}$ is a formal
symbol and
$$s_{U, \mu}^{F, G} s_{V, \nu}^{G, H}
  = \left\{ \begin{array}{ll}
  \ell(|\Gamma_\cap(U, V)|) \,
  s_{U * V, \mu * \nu}^{F, H} &
  \mbox{\rm{if $(U, \mu) \sim (V, \nu)$,}} \\
  0 & \mbox{\rm{otherwise.}}
  \end{array} \right.$$

Given $g \in G$, we define
$\Delta(G, g, G) = \{ {}^g y {\times} y :
y \in G \}$. Since
$\Delta(G, g, G) * \Delta(G, g', G) =
\Delta(G, g g', G)$ for $g' \in G$,
we have
$$s_{\Delta(G, g, G), 1}^{G, G}
  s_{\Delta(G, g', G), 1}^{G, G} =
  s_{\Delta(G, gg', G), 1}^{G, G} \; .$$
Since $\Delta(F, f, F) * U *
\Delta(G, g^{-1}, G) =
{}^{f \times g} U$ for $f \in F$, we have
$$s_{\Delta(F, f, F), 1}^{F, F}
  s_{U, \mu}^{F, G}
  s_{\Delta(G, g^{-1}, G), 1}^{G, G}
  = s_{{}^{f \times g} U,
  {}^{f \times g} \mu}^{F, G} =
  {}^{f \times g} s_{U, \mu}^{F, G} \; .$$
We make $R_\ell \cS^A$ become an
interior $R$-linear category such that
$\sigma_G(g) = s_{\Delta(G, g, G),
1}^{G, G}$. The calculations just above
confirm that $\sigma_G$ is an algebra
map and our notation is consistent.

In view of the comments we made in
Section \ref{2} concerning an $R$-basis
for $\oo{R \cP}(F, G)$, the element
$$\oo{s}_{U, \mu}^{F, G} = \sigma_F(e_F)
  s_{U, \mu}^{F, G} \sigma_G(e_G)$$
depends only on $F$, $G$ and the
$F {\times} G$-orbit $[U, \mu]$ of
$(U, \mu)$, furthermore,
$$\oo{R_\ell \cS^A}(F, G) =
  \bigoplus_{[U, \mu] \in \cS^A[F, G]}
  R \, \oo{s}_{U, \mu}^{F, G} \; .$$
To complete an explicit description of the
category $\oo{R_\ell \cS^A}$, we now
supply a formula for the composition.
By viewing $\cS^A(F, G)$ as an
$(F, G)$-biset, the
notation in the equation ${}^g (V, \nu) =
{}^{g \times 1} (V, \nu)$ makes sense for
any $g \in G$, similarly for the notation
${}^g V$ and ${}^g \nu$.

\begin{thm}
\label{4.4}
Let $F, G, H \in \cG$. Let $[U, \mu] \in
\cS^A[F, G]$ and $[V, \nu] \in
\cS^A[G, H]$. Then
$$(\oo{s}_{U, \mu}^{F, G} / |U|)
  (\oo{s}_{V, \nu}^{G, H} / |V|)
  = \frac{1}{|G|} \sum_g
  \frac{\ell(|\Gamma_\cap(U, {}^g V)|)}
  {|\Gamma_\cap(U, {}^g V)|}
  (\oo{s}_{U * {}^g V,
  \mu * {}^g \nu}^{F, H} / |U * {}^g V|)$$
where $g$ runs over representatives of
the double cosets $p_2(U) g p_1(V)
\subseteq G$ such that
$(U, \mu) \sim {}^g (V, \nu)$.
\end{thm}

\begin{proof}
By the last line of Section \ref{2},
$$\oo{s}_{U, \mu}^{F, G}
  \oo{s}_{V, \nu}^{G, H} = \frac{1}{|G|}
  \sum_y \gamma(y) \oo{s}_{U * {}^y V,
  \mu * {}^y \nu}^{F, H}$$
where $\gamma(y) =
\ell(|\Gamma_\cap(U, {}^y V)|)$ and
$y$ runs over those elements of $G$
such that $(U, \mu) \sim {}^y (V, \nu)$.
We have $(U, \mu) \sim {}^{y'} (V, \nu)$
and $\gamma(y) = \gamma(y')$ for all
$y' \in p_2(U) y p_1(V)$. So
$$\oo{s}_{U, \mu}^{F, G}
  \oo{s}_{V, \nu}^{G, H} = \frac{1}{|G|}
  \sum_g |p_2(U) g p_1(V)| \gamma(g)
  \oo{s}_{U * {}^g V,
  \mu * {}^g \nu}^{F, H} \; .$$
Since $|p_2(U) g p_1(V)|
= |p_2(U) p_1({}^g V)|$ and
$|{}^g V| = |V|$, Lemma \ref{4.2} yields
the required equality.
\end{proof}

%Section 5
\section{The fibred biset category}
\label{5}

We shall review the notion of the
{\bf $R$-linear $A$-fibred biset category}
$R \cB^A$ on $\cG$. Then we shall
introduce, more generally, an $R$-linear
category $R_\ell \cB^A$ on $\cG$.
To confirm the associativity of the
composition for $R_\ell \cB^A$, we
shall apply Theorem \ref{4.4}.

A discussion about $R \cB^A$, including
an interpretation as the $R$-linear
extension of a Grothendieck ring, can
be found in Boltje--Co\c{s}kun
\cite[Sections 1, 2]{BC18}.
We shall work with the following
characterization of $R \cB^A$.
The morphism $R$-modules are
$$R \cB^A(F, G) = \bigoplus_{[U, \mu]
  \in \cS^A[F, G]} R [(F {\times} G) /
  (U, \mu)]$$
where, for our purposes, we can regard
$[(F {\times} G) / (U, \mu)]$ as a formal
symbol uniquely determined by
the $F {\times} G$-orbit $[U, \mu]$.
See \cite[Section 1]{BC18}  for an
interpretation, not needed below, of
$[(F \times G) / (U, \mu)]$ as the
isomorphism class of an $A$-fibred
biset $(F {\times} G)/(U, \mu)$. The
composition for $R \cB^A$ is given by
$$[(F {\times} G) / (U, \mu)] .
  [(G {\times} H) / (V, \nu)] = \sum_g
  [(F {\times} H) / (U * {}^g V,
  \mu * {}^g \nu)]$$
where $g$ runs as in Theorem \ref{4.4}.
It is easy to check that the right-hand
side of the formula is well-defined,
independently of the choices of double
coset representatives $g$ and orbit
representatives $(U, \mu)$ and $(V, \nu)$.
The associativity of the composition
follows from \cite[2.2, 2.5]{BC18} or,
alternatively, Theorem \ref{5.1} below.
The identity $R \cB^A$-morphism on
$G$ is $[(G {\times} G)/(\Delta(G), 1)]$.

Generalizing, we define
$$R_\ell \cB^A(F, G) =
  \bigoplus_{[U, \mu] \in \cS^A[F, G]}
  R \, d_{U, \mu}^{F, G}$$
where $d_{U, \mu}^{F, G}$ is a formal
symbol uniquely determined by $F$,
$G$ and $[U, \mu]$. We make
$R_\ell \cB^A$ become an $R$-linear
category on $\cG$ by defining the
composition to be such that
$$d_{U, \mu}^{F, G} .
  d_{V, \nu}^{G, H} = \sum_g
  \frac{\ell(|\Gamma_\cap(U, {}^g V)|)}
  {|\Gamma_\cap(U, {}^g V)|} \,
  d_{U * {}^g V, \mu * {}^g \nu}^{F, H}$$
again with $g$ running as in Theorem
\ref{4.4}. In a moment, to confirm that
$R_\ell \cB^A$ is an $R$-linear category,
we shall make use of the ``polarization''
$R_\ell \cS^A$. We let
$$\nu_{F, G} \: : \: \oo{R_\ell \cS^A}(F, G)
  \leftarrow R_\ell \cB^A(F, G)$$
be the $R$-linear map given by
$\nu_{F, G}(d_{U, \mu}^{F, G}) =
|G| \, \oo{s}_{U, \mu}^{F, G} \, / \, |U|$.
The elements $d_{U, \mu}^{F, G}$
comprise an $R$-basis for
$R_\ell \cB^A(F, G)$, while the elements
$\oo{s}_{U, \mu}^{F, G}$ comprise an
$R$-basis for $\oo{R_\ell \cS^A}(F, G)$,
so $\nu_{F, G}$ is an $R$-isomorphism.

\begin{thm}
\label{5.1}
The composition for $R_\ell \cB^A$ is
associative and $R _\ell \cB^A$ is an
$R$-linear category on $\cG$. The
maps $\nu_{F, G}$, for $F, G \in \cG$,
determine an object-identical isomorphism
of $R$-linear categories $\nu :
\oo{R_\ell \cS^A} \leftarrow R_\ell \cB^A$.
\end{thm} 

\begin{proof}
Theorem \ref{4.4} implies that
$\nu_{F, G}(d_{U, \mu}^{F, G}) .
\nu_{G, H}(d_{V, \nu}^{G, H}) =
\nu_{F, H}(d_{U, \mu}^{F, G} .
d_{V, \nu}^{G, H})$.
By $R$-linearity, the composition
is associative. The identity
$R_\ell \cB^A$-morphism on $G$
is $d_{\Delta(G), 1}^{G, G}$.
\end{proof}

We have the following immediate
corollary, realizing $R \cB^A$ as the
invariant category not of $R \cS^A$
but of a deformation of $R \cS^A$.

\begin{cor}
\label{5.2}
Suppose $\ell(n) = n$ for all positive
integers $n$. Then there is an
object-identical isomorphism of
$R$-linear categories
$\oo{R_\ell \cS^A} \cong R \cB^A$
given by $|G| \oo{s}_{U, \mu}^{F, G}
\leftrightarrow |U| [(F {\times} G)
/ (U, \mu)]$.
\end{cor}

In \cite{BO}, it is shown that
$K_\ell \cS$ is locally semisimple
when $\ell$ satisfies the following
non-degeneracy condition: as
$q$ runs over the prime numbers,
the values $\ell(q)$ are algebraically
independent over the minimal
subfield $\QQ$ of $K$. At the time
of writing, we do not know whether
the same conclusion holds for
$K_\ell \cS^A$ under the same
non-degeneracy condition. An
approach to directly adapting the
argument in \cite{BO} would be to
make use of a suitable analogue of
\cite[3.7]{BC18}. More speculatively,
if such a generic semisimplicity
result does hold, then it might have
a bearing on the problem of
classifying the simple
$K_\ell \cS^A$-modules and, from
there, via Theorem \ref{5.1}, the
problem of classifying the simple
$K_\ell \cB^A$-modules.

\end{document}